\documentclass[11pt]{amsart}
 \usepackage{amsmath,amssymb,enumerate}
 \usepackage[all]{xy}
 \usepackage{xcolor}

 \newtheorem{pro}{Proposition}[section]
 
 \newtheorem{lem}[pro]{Lemma}

 \newtheorem{theo}[pro]{Theorem}
 \newtheorem{thm}[pro]{Theorem}

 \newcommand{\B}{\mathbb B}
 
 \newcommand{\R}{\mathbb R}
 
 \newcommand{\C}{\mathbb C}
 \newcommand{\N}{\mathbb N}

 \newcommand{\e}{\varepsilon}
 
 \newcommand{\f}{\varphi}
 
 \newcommand{\p}{\psi}

 \newcommand \de {\delta}

 \newcommand \psh {plurisubharmonic }

  \newcommand \Ec{\mathcal E }

 \numberwithin{equation}{section}

\usepackage{hyperref}
\hypersetup{
    bookmarks=true,         
    unicode=false,          
    pdftoolbar=true,        
    pdfmenubar=true,        
    pdffitwindow=false,     
    pdfstartview={FitH},    
    pdftitle={},    
    pdfauthor={},     
    colorlinks=true,       
   linkcolor=blue,          
    citecolor=blue,        
    filecolor=black,      
    urlcolor=blue}           

\begin{document}

  \title[Weak subsolutions to complex Monge-Amp\`ere equations]{Weak subsolutions to complex\\ Monge-Amp\`ere equations}

\setcounter{tocdepth}{1}

  \author{ Vincent Guedj, Chinh H. Lu, Ahmed Zeriahi}

\address{Institut de Math\'ematiques de Toulouse   \\ Universit\'e Paul Sabatier \\
118 route de Narbonne \\
F-31062 Toulouse cedex 09, France\\}

\email{vincent.guedj@math.univ-toulouse.fr}

\address{Laboratoire de Math\'ematiques d'Orsay\\
 Univ. Paris-Sud\\
 CNRS, Universit\'e Paris-Saclay\\
  91405 Orsay, France.}

\email{hoang-chinh.lu@math.u-psud.fr}

\address{Institut de Math\'ematiques de Toulouse  \\ Universit\'e Paul Sabatier \\
118 route de Narbonne \\
F-31062 Toulouse cedex 09, France\\}

\email{ahmed.zeriahi@math.univ-toulouse.fr}

 \date{\today \\
 The authors are partially supported by the ANR project GRACK}

 \begin{abstract}  
We compare various notions of weak subsolutions to degenerate complex Monge-Amp\`ere equations,
showing that they all coincide.
This allows us to give an alternative proof of mixed Monge-Amp\`ere inequalities due to Kolodziej and Dinew.
\end{abstract}

 \maketitle


\section{Introduction}
Let $\Omega$ be a domain of $\C^n$.
We consider in this article the notion of subsolution
for degenerate complex Monge-Amp\`ere equations in $\Omega$.
These are bounded \psh functions which satisfy
$$
(dd^c \f)^n \geq f dV,
$$
 where $dV$ denotes the Lebesgue measure and $0 \leq f \in L^1(\Omega)$.
 
 This inequality can be interpreted in various senses 
 (pluripotential sense \cite{BT76}, viscosity sense \cite{EGZ11}, distribution sense \cite{HL13})
 and the goal of this note is to show that all point of views do coincide.
 
  \medskip
\noindent {\bf Main theorem.}
{\it 
Assume $\f $ is \psh and bounded. The following are equivalent :

 \smallskip
 
 (i) $ (dd^c \f)^n \geq f dV$ in the pluripotential sense;
 
 \smallskip
 
 (ii) $ (dd^c (\f \star \chi_\e))^n \geq (f^{1/n} \star \chi_\e)^n dV$ in the classical sense, for all $\e>0$;
 
 \smallskip
 
  (iii) $\Delta_H \f \geq f^{1/n}$ in the sense of distributions,  for all $H \in {\mathcal H}$.
}
\medskip

\noindent The operator $dd^c=a i \partial \bar \partial$  is here normalized so that 
$d V = (dd^c \vert z\vert^2)^n$ is the euclidean volume form on $\C^n$. Thus for a smooth function $\f$,  
$$
 (dd^c \f)^n = \text{det} \left(\frac{\partial^2 \f}{\partial z_j \partial \bar {z}_k }\right) d V.
$$

We let ${\mathcal H}$ denote the space of hermitian positive definite matrix
$H$ that are normalized by $\det H=1$, and let $\Delta_H$ denote the Laplace operator 
$$
\Delta_H \f:=\frac{1}{n}\sum_{j,k=1}^n h_{jk} \frac{\partial^2 \f}{\partial z_j \partial \bar{z}_k}.
$$

The functions $\chi_\e$ are standard mollifiers, i.e. radial smooth non-negative functions with compact
support in the $\e$-ball centered at the origin, and such that $\int \chi_\e dV=1$.
It is then classical that the convolutions $\f \star \chi_\e$ are smooth, \psh, and decrease
to $\f$ as $\e$ decreases to $0$.
 
  When $f$ is moreover continuous, one can also interpret the inequality
$ (dd^c \f)^n \geq f dV$ in the viscosity sense, as shown in \cite[Proposition 1.5]{EGZ11}.

Our main theorem easily implies the following result of Kolodziej
 \cite[Lemma 1.2]{Kol03} (see also \cite{Din09,DL15}) :

   \medskip
\noindent {\bf Corollary.}
{\it 
Assume $\f_1,\ldots,\f_n $ are bounded \psh functions in $\Omega$, such that
$(dd^c \f_i)^n \geq f_i dV$, where $0 \leq f_i \in L^1(\Omega)$. Then 
$$
dd^c \f_1 \wedge \cdots \wedge dd^c \f_n \geq f_1^{1/n} \cdots f_n^{1/n} dV.
$$
 
}
\bigskip

The note is organized as follows.
We start by extending Kolodziej's subsolution theorem (see Theorem \ref{thm:SubSol}), providing a solution to special Monge-Amp\`ere equations that we are going to use in the sequel.
 We prove our main result in {\it Section \ref{sec:main}}. The starting point is an identification of viscosity subsolutions and
 pluripotential subsolutions obtained in \cite{EGZ11}. 
 We connect these identifications to mixed Monge-Amp\`ere inequalities in {\it Section \ref{sec:mixed}}
 and propose some generalizations in {\it Section \ref{sec:moregeneralrhs}}.

 \section{The subsolution theorem}
 
 Let $\Omega \Subset \C^n$ be a bounded hyperconvex domain
 (in the sequel we only need to deal with the case when $\Omega$ is a ball). 
 Let $\mu$ be a Borel measure on $\Omega$. If there exists a function $v \in PSH (\Omega) \cap L^{\infty} (\Omega)$ such that
 $$
 \mu \leq (dd^c v)^n
 \text{ in } \Omega, \;
 \text{ with }
 \;
  \lim_{\Omega \ni z \to \zeta} v (z) = 0, \, \, \forall \zeta \in \partial \Omega,
 $$
 then it was proved by S. Kolodziej  \cite{Kol95} that  there exists a
 unique  solution $\psi \in  PSH (\Omega) \cap L^{\infty} (\Omega)$ to the equation
 \[
 (dd^c \psi)^n =  \mu,  
\]
such that $\lim_{\Omega \ni z \to \zeta} \p (z) = 0$.  We need the following generalization:
 
 \begin{theo}  \label{thm:SubSol} 
 Assume $\mu$ is a non pluripolar Borel measure on $\Omega$ which has finite total mass.  
 Then there exists a unique function  $\f \in \mathcal F^1 (\Omega)$ satisfying  
 \begin{equation} \label{eq:MA}
 (dd^c \f)^n = e^{\f} \mu 
 \;  \text{ in } \; 
 \Omega.
 \end{equation}
 
Moreover if $\mu$ satisfies $\mu \leq (dd^c u)^n$ in $\Omega$, for some bounded negative psh function $u$,  then the solution  $\f \in PSH (\Omega)$  
is bounded with $ u \leq \f$. In particular, if $u\mid_{\partial \Omega}=0$ then $\f\mid_{\partial \Omega} = 0$.
\end{theo}

Before entering into the proof let us recall the definition of Cegrell's finite energy classes. We refer the reader to \cite{Ce98,Ce04} for more details. 

A domain $\Omega$ is called hyperconvex if there exists a continuous plurisubharmonic exhaustion function $\rho: \Omega \rightarrow \mathbb{R} \cup\{-\infty\}$ such that the sublevel sets $\{\rho <-c\}$ are relatively compact in $\Omega$, for all constants $c>0$.

Let $u$ be a negative plurisubharmonic function in $\Omega$. We recall the following definitions :
\begin{itemize}
	\item $u\in \Ec_0(\Omega)$ if $u$ is bounded in $\Omega$, $u$ vanishes on the boundary, i.e. $\lim_{z\to \partial \Omega}u(z)=0$, and $\int_{\Omega} (dd^cu)^n<+\infty$. 
	\item $u \in \mathcal{E}(\Omega)$  if  for each $z_0\in \Omega$ there exists an open neighborhood $z_0\in V_{z_0}\Subset \Omega$ and a decreasing sequence $(u_j)\subset \mathcal{E}_0(\Omega)$ such that $u_j$ converges to $u$ in $V_{z_0}$ and $\sup_j \int_{\Omega}(dd^c u_j)^n<+\infty$.  
	\item $u \in \Ec^p(\Omega), p>0$ if there exists a sequence $(u_j)$ in $\Ec_0(\Omega)$ decreasing to $u$ and satisfying 
\[
\sup_{j\in \N}\int_{\Omega}(-u_j)^p (dd^c u_j)^n <+\infty. 
\]
If we ask additionally that $\int_{\Omega} (dd^c u_j)^n$ is uniformly bounded then by definition $u\in \mathcal{F}^p(\Omega)$. 
\end{itemize}

It was proved in \cite{Ce98,Ce04} that the Monge-Amp\`ere operator $(dd^c)^n$ is well-defined for functions in $\mathcal{E}(\Omega)$. Moreover, it was shown in \cite[Theorem A]{BGZ09} that if $u\in \mathcal{E}(\Omega)$ then $(dd^c \max(u,-j))^n$ converges in the strong sense of Borel measures in $\Omega\cap \{u>-\infty\}$ to $(dd^c u)^n$. 
\medskip

Theorem \ref{thm:SubSol} was proved in \cite{CK06} using a fixed point argument. We provide in this note an alternative proof using the variational method, adapting the techniques developed in K\"ahler geometry in \cite{BBGZ13} (similar ideas have been used   in \cite{ACC12, Lu15}).

\begin{proof}[Proof of Theorem \ref{thm:SubSol}]
Consider 
 $$
  F_{\mu} (\phi) := E_1 (\phi)  - \int_{\Omega} e^{\phi} d \mu, \, \, \phi \in \mathcal E^{1} (\Omega),  
 $$
 where
 $$
 E_1 (\phi) := \frac{1}{n + 1} \int_{\Omega} \phi (dd^c \phi)^n.
 $$

The Euler-Lagrange equation of  $F_{\mu}$ can be computed as follows.
Fix  $\phi \in \mathcal E^{1} (\Omega)$ and assume  $(\phi (t))$ is a smooth path in $\mathcal E^{1} (\Omega)$ 
starting at $\phi (0) = \phi $ with   $\dot \phi(0) = v \in C (X)$.
It follows from Stokes theorem that
 $$
 \frac{d }{d t} E_1 (\phi (t))|_{t = 0} = \int_{\Omega} v (dd^c \phi)^n\cdot,
 $$
hence
 $$
 \frac{d }{d t} F_{\mu} (\phi (t))|_{t = 0} = \int_{\Omega}   v (dd^c \phi)^n - \int_{\Omega} v e^{\phi} d \mu.
 $$

 Thus $\phi$ is a critical point of the functional $F_{\mu}$
 if it is a solution to the complex Monge-Amp\`ere equation (\ref{eq:MA}).
 It is thus natural to try and extremize $F_{\mu}$ in order to solve (\ref{eq:MA}).
We proceed in three steps :
 
  \smallskip
  
 {\it Step 1: Upper semi-continuity of  $ F_\mu$.}
Observe first that the functional $J (\phi) := - E_1 (\phi) = \vert E_1 (\phi)\vert$ is a positive proper functional on the space $\mathcal E^1 (\Omega)$ i.e. its sublevel subsets 
$$
 \mathcal E^1_B (\Omega) := \{ \phi \in \mathcal E^1 (\Omega) ; 0 \leq J (\phi) \leq B\}, \, \, B > 0
 $$
are compact for the $L^1$-topology. Moreover the functional $E_1$ is upper semi-continuous on each compact subset $\mathcal E^1_B (\Omega) $ for the $L^1 $-topolgy.

 The  continuity of the functional $L_\mu : \phi \longmapsto  \int_\Omega e^\phi d \mu $ on  
 each compact subset $\mathcal E^1_B (\Omega) $ follows from the following  fact due to Cegrell \cite{Ce98}, \cite[Lemma 4.1]{ACC12} : if $\phi_j \to \phi$ in $\mathcal E^1_B (\Omega) $ then $\phi_j \to \phi$ $\mu$-a.e., hence by Lebesgue's convergence theorem, 
 $\lim _j L_{\mu} (\phi_j) = L_{\mu} (\phi)$ (we use here the fact that $\mu$ is non-pluripolar).
 
 This proves that $F_\mu$ is upper semi-continuous on each $\mathcal E^1_B (\Omega)$.

 \smallskip
 
 {\it Step 2: Coercivity of $ F_\mu$.}
Observe that $0 \leq e^\f \leq 1$ for $\f \in \mathcal E^1 (\Omega)$, hence
 $$
  F_\mu (\phi) \leq E_1 (\phi).
 $$
We infer that $F_\mu$ is $J-$proper on $\mathcal{E}^1 (\Omega)$, i.e.
 $$
 \lim_{J (\phi) \to + \infty }  F_{\mu} (\phi) = - \infty.
 $$ 
 This implies that the maximum  of $F_\mu$ on $\mathcal E^1 (\Omega)$ is localized at a finite level of energy,
  i.e. there exists a constant $B > 0$ such that
 $$
 \sup \{F_\mu (\phi) ; \phi \in \mathcal E^1 (\Omega)\} = \sup \{F_\mu (\phi) ; \phi \in \mathcal E^1_B (\Omega)\},
 $$

Since $F_\mu $ is upper semi-continuous on the compact set $\mathcal E^1_B (\Omega)$,  there exists 
 $\phi \in \mathcal E_B^1 (\Omega)$ which maximizes $\mathcal F_\mu$ on $  \mathcal E^1_B (\Omega)$ i.e.
 $$
 F_\mu (\phi) = \inf \{ F_\mu (\psi) ; \psi \in \mathcal E_B^1 (\Omega)\}.
  $$
  
  \smallskip
  
{\it Step 3:  $\phi$ is a critical point of $F_\mu$}. 
Fix a continuous 
test function $ \chi $ with compact support in $\Omega$ and set $ \phi (t)  := P _{\Omega} (\phi + t \chi)$ for $t \geq 0$, 
where $P_\Omega u$ denotes the plurisubharmonic envelope of $u$ in $\Omega$. 

Observe that $\phi (t) \in \mathcal E^1 (\Omega)$. Indeed let $\rho$ be a continuous psh exhaustion for $\Omega$ 
such that $\rho < \chi$ on the support of $\chi$. Then   $\phi + t \rho \leq  \phi (t)$ for $t \geq 0$ close to $0$. 
Since $\phi + t \rho \in \mathcal E^1 (\Omega)$, it follows that $\phi (t) \in \mathcal E^1 (\Omega)$.
Set 
$$
h (t) := E_1 (\phi (t)) - \int_{\Omega} e^{\phi + t \chi} d \mu.
$$ 
Then since $\phi (t) \leq \phi + t \chi$, it follows that $h (t)  \leq F_{\mu} (\phi (t)) \leq F_{\mu} (\phi)$ which means that $h$ achieves its maximum at the point $0$.

On the other hand we know by \cite{BBGZ13,ACC12,Lu15} that
$$
 \frac{d}{dt} h (t)\mid_{t = 0} = \int_\Omega \chi  (dd^c \phi)^n  - \int_\Omega \chi e^{\phi} d \mu.
$$
Since $h$ achieves its maximum at the point $0$, we have $h'(0) = 0$, hence
$$
\int_\Omega \chi  (dd^c \phi)^n  =  \int_\Omega \chi e^{\phi} d \mu.
$$ 

As the test function $\chi$ was arbitrary, 
this means that the function $\phi$ is a solution of the equation (\ref{eq:MA}).  As $\mu$ has finite total mass we actually have that $\varphi\in \mathcal{F}^1(\Omega)$.

We now prove the uniqueness. If $\psi \in \mathcal{F}^1(\Omega)$ is another solution to \eqref{eq:MA} then it follows  from the comparison principle \cite[Lemma 4.4]{Ce98} that
\begin{eqnarray*}
 \int_{\{\f < \psi\}} e^{\psi} d \mu &= & \int_{\{\f < \psi\}} (dd^c \psi)^n  \leq  \int_{\{\f < \psi\}} (dd^c \f)^n \\
 &= & \int_{\{\f < \psi\}} e^{\f} d \mu \leq \int_{\{\f < \psi\}}  e^{\psi} d \mu.
\end{eqnarray*}
We infer $\int_{\{\f < \psi\}} (e^{\psi} - e^{\varphi}) d \mu = 0$ hence $\psi \leq \f$, $\mu$-almost everywhere and $(dd^c \f)^n$-almost everywhere in $\Omega$. For each $\varepsilon>0$, since $(dd^c \varphi)^n$ vanishes in $\{\varphi\leq \psi-\varepsilon\}\subset \{\varphi<\psi\}$, it follows from \cite[Theorem 2.2]{BGZ09} that 
\[
	(dd^c \max(\varphi,\psi-\varepsilon))^n \geq  {\bf 1}_{\{\varphi>\psi-\varepsilon\}} (dd^c \varphi)^n	 = (dd^c \varphi)^n.
\]

It then follows from the comparison principle \cite[Theorem 4.5]{Ce98} that $\max(\varphi,\psi-\varepsilon) \leq \varphi$, for all $\varepsilon>0$, hence $\psi \leq \varphi$. Reversing the role of $\varphi$ and $\psi$ in the above argument gives $\varphi=\psi$, proving the uniqueness. 

\smallskip

Assume finally that  $\mu\leq (dd^c u)^n$, where $u$ is a bounded negative psh function in $\Omega$.  Then since $\varphi\leq 0$ we have $(dd^c \varphi)^n=e^{\varphi}\mu \leq \mu \leq (dd^c u)^n$. Since $u$ is bounded (in particular it belongs to the domain of definition of the complex Monge-Amp\`ere operator) and $\varphi\in \mathcal{F}^1(\Omega)$ with $(dd^c \varphi)^n$ putting no mass on pluripolar sets, it follows from \cite[Corollary 2.4]{BGZ09} that $u\leq \varphi$.  In particular $\varphi$ vanishes on the boundary $\partial \Omega$ if $u$ does so.

\end{proof}

\section{The main result}

\subsection{Proof of the main result}  \label{sec:main}

Given $\f$ a plurisubharmonic function in a domain $\Omega$, we let 
$$
\f_{\e}(z)=\f  \star \chi_{\e}(z):=\int_{\C^n} \f(z-\e w) \chi(w) dV(w)=\int_{\C^n} \f(w) \chi_{\e}(z - w) dV(w)
$$
denote the standard regularizations of $\f$ defined in $\Omega_\e$ for $\e > 0$ small enough,
where
$\Omega_\e=\{ z \in \Omega , \; {\rm dist}(z,\partial \Omega)>\e\}$.

Here $\chi_{\e}$ are  non-negative radial functions with compact support in the ball $\B(\e)$ of radius $\e$
and such that $\int_{\C^n} \chi_{\e} dV=1$, where $dV$ denotes the euclidean volume form.
The first expression shows that $\f_{\e}$ is a (positive) sum of plurisubharmonic functions (hence  itself plurisubharmonic) in $\Omega_\e$,
while the second expression shows that $\f_{\e}$ is smooth in $\Omega$.

\subsubsection{The implication $(ii) \Rightarrow (i)$}

The mean value property shows that the $\f_{\e}$'s
decrease to $\f$ as $\e$ decreases to zero. It follows therefore from Bedford-Taylor's continuity results \cite{BT76,BT82}
that   $(ii) \Rightarrow (i)$ holds.

\subsubsection{The equivalence $(ii) \Leftrightarrow (iii)$}

The starting point of $(iii)$ is the classical interpretation of the determinant as
an infimum of traces :

\begin{lem} \label{lem:gav}
$$
(\det Q)^{1/n}=\inf\{ n^{-1} \rm{tr } (HQ) \, ; \, H \in {\mathcal H} \}.
$$
 \end{lem}

 We first show that $(iii) \Rightarrow (ii)$.
Indeed assume that
$$
\Delta_H \f \geq f^{1/n}
$$
for all positive definite hermitian matrix $H$ normalized by $\det H=1$.
Since $\Delta_H$ is a linear operator, we infer
$$
\Delta_H (\f \star \chi_{\e} ) \geq 
f^{1/n} \star \chi_{\e}.
$$
This inequality holds for all normalized $H$, hence Lemma \ref{lem:gav} yields
$$
(dd^c (\f \star \chi_{\e}))^n \geq (f^{1/n} \star \chi_{\e})^n dV,
$$
 where this inequality holds in the classical (pointwise, differential) sense.
 
 We conversely check that $(ii) \Rightarrow (iii)$.
 Since $\f \star \chi_{\e}$ is smooth,  Lemma \ref{lem:gav} shows indeed that
 $$
\Delta_H (\f \star \chi_{\e}) \geq 
f^{1/n} \star \chi_{\e}
$$
 for all normalized $H \in {\mathcal H}$. Letting $\e \rightarrow 0$ 
 and taking limits in the sense of distributions yields $(iii)$.

\subsubsection{The implication $(i) \Rightarrow (ii)$}

We finally focus on the most delicate implication.

{\it Step 1}.
Assume first that $(dd^c \f)^n \geq f dV$, with $f$ continuous.
This inequality can be here interpreted equivalently in the pluripotential or in the viscosity sense,
as shown in \cite[Proposition 1.5]{EGZ11}, whose proof moreover shows the equivalence with
the property that
$$
\Delta_H \f \geq f^{1/n} 
\text{  for all } H \in {\mathcal H}.
$$
Thus $(i) \Leftrightarrow (iii)$ in our main theorem, when $f$ is continuous.
Since any lower semi-continuous function is the increasing limit of continuous functions,
the implication $(i) \Rightarrow (iii)$ immediately extends to the case when $f$ is lower semi-continuous.

It remains to get rid of this extra continuity assumption.
We  are going to approximate $f$ by continuous densities $f_k$, use 
the previous result and stability estimates to conclude.
The approximation process, inspired by \cite{Ber13}, is somehow delicate, so we proceed in several steps.

\smallskip

{\it Step 2.} 
Note first  that we can assume that $f$ is {\it bounded} :
we can replace $f$ by $\min(f,A) \in L^{\infty}(\Omega)$
and let eventually $A$ increase to $+\infty$.
Since the problem is local, we can work on fixed balls $B'\Subset B$ and 
use a max construction to modify $\f$ in a neighborhood of the boundary $\partial B$,
 making it equal to the defining function of $B$.

We fix $0 < \de < 1$ and $j \in \N^*$.
Since $f \in L^2 (B) \supset L^{\infty}(B)$, it follows from \cite{CP92,Kol95} that there exists $U_f \in PSH (B) \cap C^0 (\bar B)$ such that 
$$
(dd^ c U_f)^ n = f d V, \, \, \, U_f= 0 \, \, \, \text{in} \, \, \, \partial B.
$$
Set $C_j := \sup e^{- j\f \slash n}$ and observe that
 $$
  e^ {- j \f}  \left\{f dV+\de (dd^c \f)^n\right\} \leq (dd^c v)^n
 $$
where $v := C_j (U_f + \f)$ is bounded, plurisubharmonic, with $v = 0$ on $\partial B$. 

 By Theorem \ref{thm:SubSol}  there exists a unique bounded  \psh solution $\f_{j,\delta}$ to 
the Dirichlet problem
\begin{equation} \label{eq:kol}
(dd^c \f_{j,\delta})^n=e^{j(\f_{j,\delta}-\f)} \left\{ f dV+\de (dd^c \f)^n \right\}
\end{equation}
in $B$ with boundary values $0$.

 We now observe that $\f_{j,\delta}$ uniformly converges to $\f$, as $j \rightarrow +\infty$, independently of 
 the value of $\delta>0$ :

\begin{lem} \label{lem: approx1}
For all $j \geq 1,\de \in (0,1) ,z \in B$,
$$
\f(z)- \frac{\log (1+ \delta)}{j} \leq \f_{j,\de}(z) \leq \f(z)+ \frac{(-\log \de)}{j}.
$$
\end{lem}

\begin{proof}
It follows from the comparison principle that 
$\f_{j,\delta}$ is the envelope of subsolutions.
 It thus suffices to find good sub/supersolutions
to insure that $\f_{j,\delta}$ converges to $\f$, as $j \rightarrow +\infty$.

Observe that $u=\f-\frac{\log (1+\de)}{j} \leq \f$ is \psh in $B$, with boundary values
$u_{| \partial B} \leq 0$. Moreover
$$
(dd^c u)^n=(dd^c \f)^n=e^{j(u-\f)} (1+\de) (dd^c \f)^n
\geq e^{j(u-\f)} \left\{ f dV+\de (dd^c \f)^n \right\},
$$
since $(dd^c \f)^n \geq f dV$.
Thus $u$ is a subsolution to the Dirichlet problem, showing that $u \leq \f_{j,\delta}$.

Set now $v=\f+ \frac{(-\log \de)}{j}$.
This is a \psh function in $B$ such that $v \geq 0$ on $\partial B$ and
$$
(dd^c v)^n=(dd^c \f)^n=e^{j(v-\f)} \de  (dd^c \f)^n
\leq e^{j(v-\f)} \left\{ f dV+\de (dd^c \f)^n \right\}.
$$
Thus $v$ is a supersolution of the Dirichlet problem hence $\f_{j,\de} \leq v$.
\end{proof}

\smallskip

{\it Step 3.}
 We now approximate $f$ in $L^2$ by continuous  densities $0 \leq f_k$, with
 $||f_k-f||_{L^2} \rightarrow 0$ as $k \rightarrow +\infty$.
 Extracting and relabelling, we can assume that there exists
 $g \in L^2 (B)$ such that $f_k \leq g$   for all $k \in \N$ and $f_k$ converges almost everywhere to $f$. 
Arguing as above we obtain
$$
 e^ {- j \f}  \left\{f_k dV+\de (dd^c \f)^n\right\} \leq (dd^c v)^n,
 $$
 where $v := C_j (U_g + \f)$ is bounded, plurisubharmonic, with $v = 0$ in $\partial B$. 
By Theorem \ref{thm:SubSol},
there exists a unique bounded  \psh solution $\f_{j,\delta,k}$ to the Dirichlet problem
$$
(dd^c \f_{j,\delta,k})^n=e^{j(\f_{j,\de,k}-\f)} \left\{ f_k dV+\de (dd^c \f)^n \right\}
$$
in $B$, with zero  boundary values. 

The comparison principle shows that for all $k \in \N$,
\begin{equation}\label{eq: uniform-est0}
C_j (U_g + \f) \leq  \f_{j,\de,k} \leq 0.	
\end{equation}
Thus $k \longmapsto \f_{j,\de,k}$ is uniformly bounded in $B$.
 Extracting and relabelling, we can assume that   it converges 
to a plurisubharmonic function $\psi = \psi_{j,\delta}$ in $L^1 (\Omega)$  such that 
 \begin{equation} \label{eq: uniform-est}
 C_j (U_g + \f) \leq  \psi_{j,\delta} \leq 0.
 \end{equation}

We claim that $\psi_{j,\delta} = \f_{j,\delta}$ in $B$.
To simplify notations we write $u_k := \varphi_{j,\delta, k}$ and $u := \psi_{j,\delta}$. 
From  \eqref{eq: uniform-est0} and \eqref{eq: uniform-est}, it follows that  $u_k =u= 0$ in $\partial B$. 
On the other hand let $\tilde u_\ell := (\sup_ {k \geq \ell} u_k)^*$ for $\ell \in \N$. 
This is a decreasing sequence of bounded \psh functions converging to $u$ in $B$. 
We infer for all $\ell$,
$$
(dd^c \tilde u_\ell)^n \geq e^{\inf_{k \geq \ell} j (u_k - \f)} \inf_{k \geq \ell} (f_k d V + \delta (dd^c \f)^n).
$$
Letting $\ell \rightarrow +\infty$ yields
$$
(dd^c u)^n \geq e^{ j (u - \f)} \left\{f d V + \delta (dd^c \f)^n \right\},
$$
which implies that $u = \psi_{j,\delta}$ is a subsolution to the Dirichlet problem for the equation (\ref{eq:kol}).
Hence $\psi_{j,\delta} \leq \f_{j,\delta}$.

By \cite{Kol95} there exists a bounded plurisubharmonic function  $\rho_k$ in $B$, solution to the   Dirichlet problem
$$
 (dd^c \rho_k)^n = e^{j(\f_{j,\de}-\f)} \vert f - f_k \vert dV, 
 \text{ with } {\rho_k}_{|\partial B} =0,
$$
with the   uniform bound
$$
\Vert \rho_k \Vert_{L^{\infty} (B)} \leq C \Vert f - f_k\Vert_{L^2 (B)}^{1 \slash n},
$$ 
where $C > 0$ is independent of $k$. In particular $\rho_k \to 0$ uniformly in $B$.

Since $ f_k \leq f + \vert f - f_k\vert$ and $\rho_k\leq 0$ it follows that  $w := \f_{j,\delta} + \rho_k $ satisfies
$$
(dd^c w)^n = (dd^c (\f_{j,\delta} + \rho_k))^n \geq e^{j(w-\f)} (f_k dV+ \delta (dd^c \f)^n).
$$
The comparison principle insures 
$w \leq \f_{j,\delta,k}$ hence
$\f_{j,\delta} \leq \psi_{j,\delta}$ since $\rho_k \to 0$.

\smallskip

{\it Conclusion.}
We have thus shown that $ \psi_{j,\delta} = \f_{j,\delta}$  and $\f_{j,\delta,k}  \geq \f_{j,\delta}  + \rho_k$.
 Lemma~\ref{lem: approx1} yields 
 $$
 j (\f_{j,\delta,k} - \f) \geq - \log (1 + \delta) - j \eta_k, 
 $$
 where  $\eta_k := \Vert \rho_k \Vert_{L^{\infty} (B)} \stackrel{k \rightarrow +\infty}{\rightarrow} 0$.

Since $f_k$ is continuous we can   apply {\it Step 1} to insure that 
 $$
 (dd^c \f_{j,\delta,k} \star\chi_{\e})^n \geq \frac{e^{- j \eta_k}}{\delta + 1} (f_k^{1/n} \star\chi_{\e})^n dV.
 $$
 
We know that $\limsup_{k \to + \infty} \f_{j,\delta,k} \leq \f_{j,\delta}$ since  $\f_{j,\delta,k} \to \f_{j,\delta}$ in $L^1(B)$  as $k \to + \infty$.
Since  $\f_{j,\delta,k}  \geq \f_{j,\delta}  - \eta_k $ and $\lim_{k \to + \infty} \eta_k = 0$, it follows from Hartogs lemma that 
$\f_{j,\delta,k} \to \f_{j,\delta}$ in capacity. Letting $k \to + \infty$ we   obtain
$$
 (dd^c \f_{j,\delta} \star \chi_\e)^n \geq \frac{1}{\delta + 1} (f^{1/n} \star\chi_\e)^n dV.
 $$
 By Lemma~\ref{lem: approx1}  $(\f_{j,\delta})$ uniformly converges  to $\f$ as $j \to + \infty$, hence 
 $$
 (dd^c \f \star\chi_\e)^n \geq \frac{1}{\delta + 1} (f^{1/n} \star\chi_\e)^n dV.
 $$
 
 We let finally $\de$ decrease to zero to obtain the desired lower bound $(ii)$.

\subsubsection{An extension of the main result}
 By approximating a given function $\f$ in the Cegrell class $\mathcal E (\Omega)$ by the decreasing sequence  $\f_j :=  \max (\f,- j)$  of bounded plurisubharmonic functions, we let the reader check that the main theorem holds when $\f$ merely  belongs to $\mathcal E (\Omega)$.

\subsection{Mixed inequalities}  \label{sec:mixed}

We now prove the Corollary of the introduction on mixed Monge-Amp\`ere measures of subsolutions, 
providing an alternative proof of  \cite[Lemma 1.2]{Kol03}:

\begin{pro}
Assume $\f_1,\ldots,\f_n $ are bounded \psh functions in $\Omega$, such that
$(dd^c \f_i)^n \geq f_i dV$, where $0 \leq f_i \in L^1(\Omega)$. Then 
$$
dd^c \f_1 \wedge \cdots \wedge dd^c \f_n \geq f_1^{1/n} \cdots f_n^{1/n} dV.
$$
\end{pro}

\begin{proof}
The inequality is classical when the functions $\f_i$ are smooth, and follows from the concavity
of $H \mapsto \log \det H$ (see \cite[Corollary 7.6.9]{HJ85}).

To treat the general case we replace each $\f_i$ by its convolutions $\f_i \star\chi_\e$.
We can always assume that $f_i \in L^{\infty} (\Omega)$.
Our main result insures that 
\[
(dd^c (\f_i \star \chi_\e))^n \geq (f_i^{1 \slash n} \star\chi_\e)^n dV,
\]
hence 
$$
dd^c (\f_1 \star \chi_\e) \wedge \cdots \wedge dd^c (\f_n \star \chi_\e) 
\geq (f_1^{1 \slash n} \star \chi_\e) \cdots (f_n^{1 \slash n} \star \chi_\e) dV.
$$

The left hand side converges weakly to $dd^c \f_1 \wedge \cdots \wedge dd^c \f_n$ by
Bedford-Taylor's continuity results \cite{BT76,BT82}, while 
$(f_i^{1 \slash n} \star \chi_\e)$ converges to $(f_i)^{1/n}$ in $L^n$ by Lebesgue convergence Theorem. Hence
$$
(f_1 \star \chi_\e)^{1/n} \cdots (f_n \star \chi_\e)^{1/n}
\text{  converges to }
(f_1)^{1/n} \cdots (f_n)^{1/n}
$$
 in $L^1$.
The conclusion follows.
\end{proof}

  We note conversely that these mixed inequalities yield an important implication in our main result.
  Assume indeed that $(dd^c \f)^n \geq f dV$ in the pluripotential sense.
Fix $f_1 = f$ and $\f_2=\ldots=\f_n=\rho_H$, where 
$\rho_H=\sum h_{jk} z_j \bar{z}_k$ with $H \in {\mathcal H}$, so that 
$(dd^c \f_i)^n \geq f_i dV$ with $f_2=\ldots=f_n=1$. 
It follows from the mixed inequalities above that
$$
\Delta_H \f = dd^c \f \wedge dd^c \f_2 \wedge \cdots \wedge dd^c \f_n \geq f^{1/n}.
$$

 One can alternatively proceed as follows : observe that
 \begin{eqnarray*}
 \lefteqn{(dd^c \f \star\chi_\e)^n(z) } \\
 & = &\int dd^c \f(z-w_1) \wedge \cdots \wedge dd^c \f(z-w_n)\chi_\e(w_1) \cdots\chi_\e(w_n) dV(w_1,\ldots w_n)  \\
 &\geq& \int f^{1/n}(z-w_1) \cdots f^{1/n}(z-w_n) dV(z)\chi_\e(w_1) \cdots\chi_\e(w_n) dV(w_1,\ldots w_n) \\
 &=& (f^{1/n} \star\chi_\e)^n (z) dV(z).
  \end{eqnarray*}

 \subsection{More general right hand side}  \label{sec:moregeneralrhs}
 
 There are several ways one can
  extend  our main observation. We note here the following:

 \begin{thm} \label{thm:moregeneralrhs}
 Assume  $\f $ is \psh and bounded.
  Fix  $g \in L^1$ and $h: \R \rightarrow \R$ convex.
 The following are equivalent :

 \smallskip
 
 (i) $ (dd^c \f)^n \geq e^{h(\f)+g} dV$ in the pluripotential sense;
 
 \smallskip
 
 (ii) $ (dd^c (\f \star \chi_\e))^n \geq e^{h(\f \star \chi_\e)+g \star \chi_\e}  dV$
 in the classical sense, for all $\e>0$;
 
 \smallskip
 
  (iii) $\Delta_H \f \geq  e^{h(\f)/n+g/n} $ in the sense of distributions,  for all $H \in {\mathcal H}$.
 \end{thm}

 The proof is very close to what we have done above, using  the convexity of $\exp$ and $h$
 through Jensen's inequality. 
 We leave the details to the reader.

\end{document}